\RequirePackage{fix-cm}
\documentclass[smallextended]{svjour3}       
\smartqed  
\usepackage{graphicx}
\usepackage{mathptmx}      
\usepackage{amsmath}
\usepackage{amssymb}
\usepackage{nicefrac}
\usepackage{caption}
\usepackage{subcaption}

\newtheorem{assumption}{Assumption}%

\usepackage{xcolor}
\usepackage{hyperref}
\hypersetup{colorlinks=true, linkcolor=blue, citecolor=blue}
\makeatletter
\def\cl@chapter{\@elt {chapter}}
\makeatother
\usepackage[noabbrev,capitalize,nameinlink]{cleveref}

\crefformat{equation}{(#2#1#3)}
\crefrangeformat{equation}{(#3#1#4) to~(#5#2#6)}
\crefmultiformat{equation}{(#2#1#3)}{ and~(#2#1#3)}{, (#2#1#3)}{ and~(#2#1#3)}
\crefname{assumption}{Assumption}{Assumptions}
\newcommand{\R}{\mathbb{R}}
\newcommand{\C}{\mathbb{C}}
\newcommand{\F}{\mathcal{F}}
\newcommand{\Hc}{\mathcal{H}}

\newcommand{\sgn}{\mathrm{sgn}}
\newcommand\restr[2]{\ensuremath{\left.#1\right\vert_{#2}}}
\newcommand\abs[1]{\ensuremath{\left.\vert#1\right\vert}}
\newcommand\tensor{\,\tilde{\otimes}\,}

\journalname{}

\begin{document}

\title{Time-dependent Steklov--Poincar\'e operators and space-time Robin--Robin decomposition for the heat equation
\thanks{This work was supported by the Swedish Research Council under the grant 2019--05396.}
}
\thanks{This work was supported by the Swedish Research Council under the grant 2019--05396.}
\titlerunning{Space-time Robin--Robin method}

\author{Emil Engström        \and
             Eskil Hansen%
}

\institute{
  Emil Engström \at
  Centre for Mathematical Sciences, Lund University, P.O.\ Box 118, 221 00 Lund, Sweden \\
  \email{emil.engstrom@math.lth.se}
  \and
  Eskil Hansen \at
  Centre for Mathematical Sciences, Lund University, P.O.\ Box 118, 221 00 Lund, Sweden \\
  \email{eskil.hansen@math.lth.se}
}
\date{\today}
\maketitle

\begin{abstract}
Domain decomposition methods are a set of widely used tools for parallelization of partial differential equation solvers. Convergence is well studied for elliptic equations, but in the case of parabolic equations there are hardly any results for general Lipschitz domains in two or more dimensions. The aim of this work is therefore to construct a new framework for analyzing nonoverlapping domain decomposition methods for the heat equation in a space-time Lipschitz cylinder. The framework is based on a variational formulation, inspired by recent studies of space-time finite elements using Sobolev spaces with fractional time regularity. In this framework, the time-dependent Steklov--Poincaré operators are introduced and their essential properties are proven. We then derive the interface interpretations of the Dirichlet--Neumann, Neumann--Neumann and Robin--Robin methods and show that these methods are well defined. Finally, we prove convergence of the Robin--Robin method and introduce a modified method with stronger convergence properties.
\end{abstract}

\newpage

\section{Introduction}\label{sec:intro}
Domain decomposition methods enable the usage of parallel and distributed hardware and are commonly employed when approximating the solutions to elliptic equations. The basic idea is to first decompose the equation's spatial domain into subdomains. The numerical method then consists of iteratively solving the elliptic equation on each subdomain and thereafter communicating the results via the boundaries to the neighboring subdomains. An in-depth survey of the topic can be found in the monographs~\cite{quarteroni,widlund}. 

A recent development in the field is to apply this approach to parabolic equations. The decomposition into spatial subdomains is then replaced by the decomposition into space-time cylinders. In general, space-time decomposition schemes enable additional parallelization and less storage requirements when combined with a standard numerical method for parabolic problems. The methods have especially gained attention in the contexts of parallel time integrators; surveyed in~\cite{gander15}, space-time finite elements; surveyed in~\cite{steinbach19}, and parabolic problems with a spatial domain given by a union of domains with very different material properties~\cite{japhet20,japhet16}.

There have been several studies concerning the convergence and other theoretical aspects of space-time decomposition schemes applied to parabolic equations on one-dimensional or rectangular spatial domains; see~\cite{gander07,kwok16,kwok21,keller02,lemarie13a,lemarie13b}. However, there are hardly any convergence results for more general settings, e.g., with spatial Lipschitz domains in higher dimensions. The one exception is the convergence analysis considered in~\cite{japhet13} for an optimized Schwarz waveform relaxation, where a method-specific variational approach is proposed. The latter is similar to the analysis used by Lions for the Robin--Robin method applied to elliptic equations~\cite{lions3}. 

This lack of convergence results is in stark contrast to the elliptic setting, where the analysis of nonoverlapping domain decomposition schemes for Lipschitz domains in higher dimensions is standard. The basic elliptic framework is to reformulate the method as an iteration scheme posed on the intersections of the subdomains in terms of Steklov--Poincar\'e operators. This is the basis for the analysis of the Dirichlet--Neumann and Neumann--Neumann methods~\cite{quarteroni}. The Steklov--Poincar\'e framework even yields the convergence of the Robin--Robin method applied to nonlinear elliptic equations~\cite{EHEE22}. Note that the time-dependent Steklov--Poincar\'e operators were introduced without any analysis in~\cite{japhet13}.  

Hence, the aim of this study is twofold. First, we strive to introduce a new variational framework for time-dependent Steklov--Poincar\'e operators and to derive the essential properties of these operators, e.g., coercivity and bijectivity. The latter property implies that the space-time generalizations of the Dirichlet--Neumann, Neumann--Neumann, and Robin--Robin methods are all well defined in the parabolic setting. Second, we will employ the new framework to give a first rigorous proof of the Robin--Robin method's convergence when applied to parabolic equations. 

For sake of simplicity, we will restrict our attention to the heat equation with homogeneous Dirichlet conditions. However, our analysis does not depend on any symmetry of the underlying bilinear form, and a similar framework can be used for more general parabolic problems and boundary conditions. To limit the technicalities further, we consider the equation for all times $t\in\R$, which allows a variational formulation with the same test and trial spaces. Our model problem is therefore the heat equation on the space-time cylinder $\Omega\times\R$, i.e.,
\begin{equation}\label{eq:strong}
\left\{
     \begin{aligned}
            (\partial_t-\Delta) u&=f  & &\text{in }\Omega\times\R,\\
             u&=0 & &\text{on }\partial\Omega\times\R,
        \end{aligned}
\right.
\end{equation}
where the spatial domain $\Omega\subset\R^d$, $d=2,3$, is bounded with boundary $\partial\Omega$. Note that a homogeneous initial condition at time $t_0$ can be incorporated by prescribing a source term $f$ that is zero for times $t<t_0$.

Next, we decompose the spatial domain $\Omega$ into nonoverlapping subdomains $\Omega_i$, $i=1,2$, with boundaries $\partial\Omega_i$, and denote the interface separating the subdomains $\Omega_i$ by $\Gamma$. That is, 
\begin{equation}\label{eq:domain}
\overline{\Omega}=\overline{\Omega}_1\cup\overline{\Omega}_2,\quad \Omega_1\cap\Omega_2=\emptyset,\quad\text{and}\quad\Gamma=(\partial\Omega_1\cap\partial\Omega_2)\setminus\partial\Omega. 
\end{equation}
The space-time cylinder $\Omega\times\R$ is thereby decomposed into $\Omega_i\times\R$, $i=1,2$, as illustrated in~\cref{fig:spacetimecylinder}. The current setting is also valid for spatial subdomains $\Omega_i$ given as unions of nonadjacent subdomains, i.e., 
\begin{displaymath}
\Omega_i=\cup_{\ell=1}^s\Omega_{i\ell}\quad\text{and}\quad
\overline{\Omega}_{i\ell}\cap\overline{\Omega}_{ij}=\emptyset\quad\text{for }\ell\neq j. 
\end{displaymath}

\begin{figure}
\centering
\includegraphics[width=.5\linewidth]{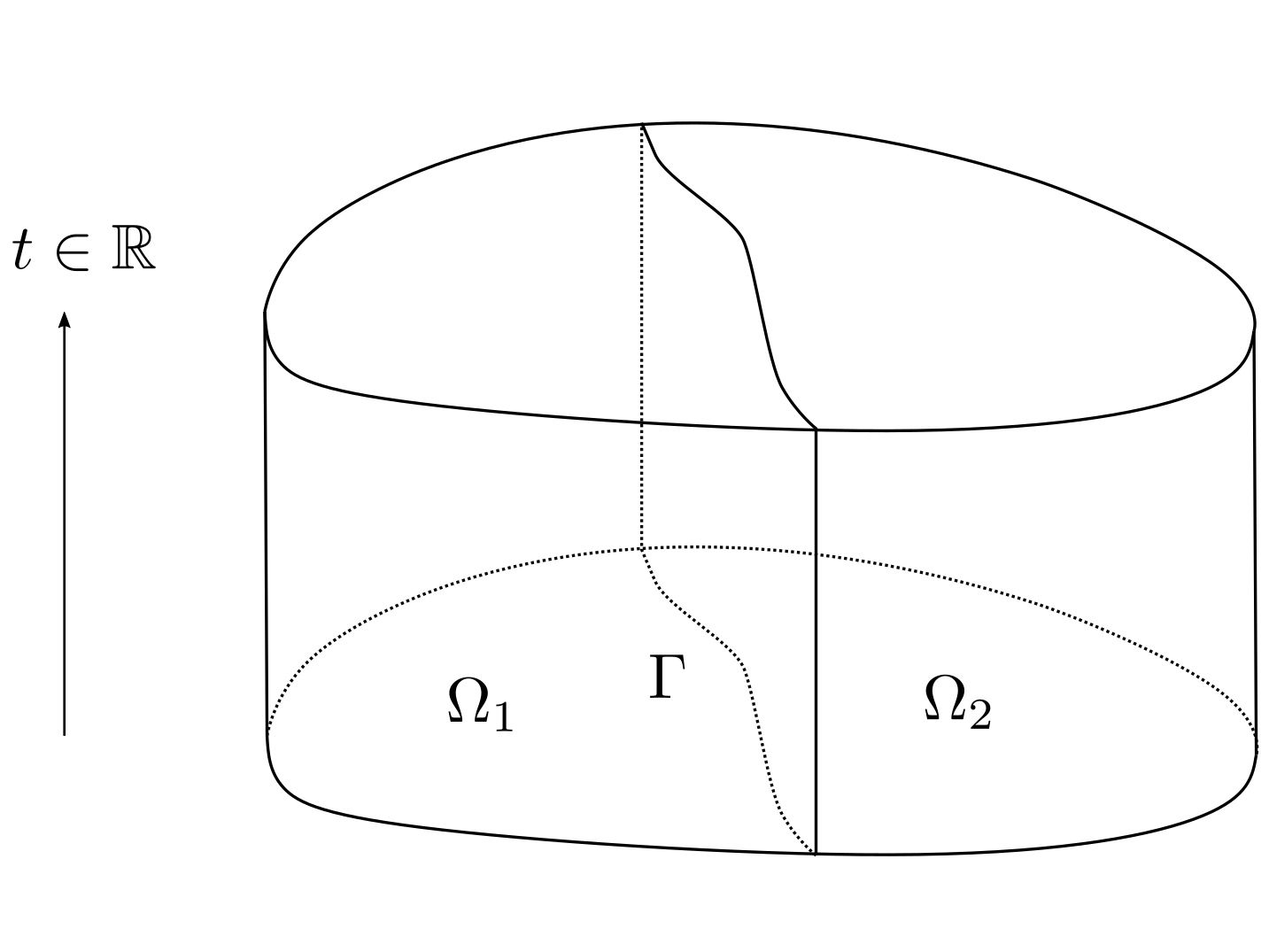}
\caption{The nonoverlapping decomposition of the space-time cylinder.}
\label{fig:spacetimecylinder}
\end{figure}%
With a fixed domain decomposition we can reformulate the heat equation as equations on $\Omega_i\times\R$, $i=1,2$, connected via transmission conditions on $\Gamma\times\R$. More precisely, we have the parabolic transmission problem
\begin{equation}\label{eq:TPStrong}
\left\{
     \begin{aligned}
            (\partial_t-\Delta) u_i&=f  & &\text{in }\Omega_i\times\R & &\text{ for }i=1,2,\\
             u_i&=0 & &\text{on }(\partial\Omega_i\setminus\Gamma)\times\R & &\text{ for }i=1,2,\\[5pt]
             u_1&=u_2  & &\text{on }\Gamma\times\R, & &\\
             \nabla u_1\cdot\nu_1 &= -\nabla u_2\cdot\nu_2  & &\text{on }\Gamma\times\R, & &
        \end{aligned}
\right.
\end{equation}
where $\nu_i$ denotes the unit outward normal vector of $\partial\Omega_i$. Alternating between the decomposed space-time cylinders and the transmission conditions generates the space-time generalizations of the Dirichlet--Neumann and Neumann--Neumann methods. As $\nu_1=-\nu_2$ on $\Gamma$, the transmission conditions are also equivalent to the Robin conditions 
\begin{displaymath}
\nabla u_1\cdot\nu_i+su_1=\nabla u_2\cdot\nu_i+su_2\quad\text{on }\Gamma\times\R\text{ for }i=1,2,
\end{displaymath}
where $s$ is a non-zero parameter. The latter reformulation gives rise to the space-time Robin--Robin method, for which we will prove convergence. The method is given by finding $(u^n_1,u^n_2)$ for $n=1,2,\ldots$ such that
\begin{equation}\label{eq:RobinStrong}
\left\{
     \begin{aligned}
            (\partial_t-\Delta) u^{n+1}_1&=f  & &\text{in }\Omega_1\times\R,\\
             u^{n+1}_1&=0 & &\text{on }(\partial\Omega_1\setminus\Gamma)\times\R,\\
             \nabla u^{n+1}_1\cdot\nu_1 + s u^{n+1}_1 &=  \nabla u^n_2\cdot\nu_1 + s u^n_2 & &\text{on }\Gamma\times\R,\\[5pt]
 	     (\partial_t-\Delta)u^{n+1}_2&=f  & &\text{in }\Omega_2\times\R,\\
             u^{n+1}_2&=0 & &\text{on }(\partial\Omega_2\setminus\Gamma)\times\R,\\
             \nabla u^{n+1}_2\cdot\nu_2 + s u^{n+1}_2 &=\nabla u^{n+1}_1\cdot\nu_2 + s u^{n+1}_1 & &\text{on }\Gamma\times\R,
        \end{aligned}
\right.
\end{equation}
where $u_2^0$ is a given initial guess, $s>0$ is a fixed method parameter, and $u^n_i$ approximates $u|_{\Omega_i\times\R}$. Note that the method per se is sequential, but the computation of each term $u^n_i$ can be implemented in parallel if $\Omega_i$ is a union of nonadjacent subdomains.

The standard variational framework for parabolic problems, based on solutions in the space 
\begin{displaymath}
H^1\bigl(\R,H^{-1}(\Omega_i)\bigr)\cap L^2\bigl(\R,H^{1}(\Omega_i)\bigr),
\end{displaymath}
is unfortunately not well suited for our domain decompositions. The issue is that two functions $u_i$, $i=1,2$, in the above solution space, which coincide on $\Gamma\times\R$ in the sense of trace, 
can not be ``pasted'' together into a new function in $H^1\bigl(\R,H^{-1}(\Omega)\bigr)$; compare with~\cite[Example~2.14]{costabel90}. Hence, the transmission problem~\cref{eq:TPStrong} does not necessarily yield a solution to the heat equation~\cref{eq:strong} in this context. 

In order to remedy this, we consider a framework with solutions in 
\begin{displaymath}
H^{1/2}\bigl(\R,L^2(\Omega_i)\bigr)\cap L^2\bigl(\R,H^{1}(\Omega_i)\bigr),
\end{displaymath}
which originates from~\cite{lionsmagenes2} and resolves the above issue. This $H^{1/2}$-setting also enables a trace theory valid for spatial Lipschitz domains~\cite{costabel90} and the related bilinear form becomes coercive-equivalent~\cite{fontesthesis}. This will be the starting point for our analysis. Note that the $H^{1/2}$-setting has been used by~\cite{costabel90} in the context of boundary element methods and was later employed for space-time finite elements~\cite{fontesthesis,larsson15,steinbach20}. 

The analysis is organized as follows: In~\cref{sec:prel,sec:tensor} we derive the properties of the required function spaces and operators. We introduce the variational formulations and prove equivalence between the heat equation and the transmission problem in~\cref{sec:weak}. The time-dependent Steklov--Poincar\'e operators are analysed in~\cref{sec:SP} and three standard space-time domain decompositions are proven to be well defined. In~\cref{sec:conv,sec:convmod} we prove convergence of the space-time Robin--Robin method in $L^2\bigl(\R,H^{1}(\Omega_1)\bigr)\times L^2\bigl(\R,H^{1}(\Omega_2)\bigr)$ and the convergence of a modified version of the method in
\begin{displaymath}
H^{1/2}\bigl(\R,L^2(\Omega_1)\bigr)\cap L^2\bigl(\R,H^{1}(\Omega_1)\bigr)\,\times\,
H^{1/2}\bigl(\R,L^2(\Omega_2)\bigr)\cap L^2\bigl(\R,H^{1}(\Omega_2)\bigr),
\end{displaymath}
respectively.

\section{Preliminaries}\label{sec:prel}
The spatial domain $\Omega\subset \R^{d}$ and its decomposition~\cref{eq:domain} into $\Omega_{i}$, $i=1,2$, is assumed to satisfy the properties below.
\begin{assumption}\label{ass:domains}
The sets $\Omega,\,\Omega_i,\,i=1,2$, are bounded and Lipschitz. The spatial interface~$\Gamma$ and the sets $\partial\Omega\setminus\partial\Omega_{i}$, $i=1,2$, are all $(d-1)$-dimensional Lipschitz manifolds.
\end{assumption}
For a description of Lipschitz domains, see~\cite[Chapter 6.2]{kufner}. The assumptions are made in order ensure the existence of the spatial trace operator, as well as, to allow the usage of Poincar\'e's inequality. As a notational convention, operators only depending on space or time are ``hatted'' and their extensions to space-time are denoted without hats, e.g., 
\begin{displaymath}
\hat{\Delta}:H_0^1(\Omega)\rightarrow H^{-1}(\Omega)\quad\text{and}\quad \Delta:L^{2}\bigl(\R,H_0^1(\Omega)\bigr)\rightarrow L^{2}\bigl(\R,H^{-1}(\Omega)\bigr).
\end{displaymath}
Furthermore, $c$ and $C$ denote generic positive constants.

Consider the spatial function spaces
\begin{displaymath}
V=H_0^1(\Omega),\quad  V_i^0=H_0^1(\Omega_i),\quad\text{and}\quad
V_i=\{v\in H^1(\Omega_i): \restr{(\hat{T}_{\partial\Omega_i}v)}{\partial\Omega_i\setminus\Gamma}=0\}.
\end{displaymath}
Here, $\hat{T}_{\partial\Omega_i}: H^1(\Omega_i)\rightarrow H^{1/2}(\partial\Omega_i)$ denotes the spatial trace operator, see~\cite[Theorem 6.8.13]{kufner}. The norm on $V_i$ and $V_i^0$ is given by 
\begin{displaymath}
	    \|v\|_{V_i}=\bigl(\|\nabla v\|^2_{L^2(\Omega_i)^{d}}+\|v\|^2_{L^2(\Omega_i)}\bigr)^{1/2}.
\end{displaymath}
By Poincaré's inequality and \cref{ass:domains} we have that $v\mapsto\|\nabla v\|_{L^2(\Omega_i)^{d}}$ is an equivalent norm on $V_i$ and $V_i^0$. 
The Hilbert spaces $V$, $V_i$, and $V_i^0$ are all separable. The space $H^{1/2}(\partial\Omega_i)$ is defined as
\begin{equation}\label{eq:slobodetskii}
	\begin{aligned}
	    &H^{1/2}(\partial\Omega_i)=\{v\in L^2(\partial\Omega_i):  \|v\|_{H^{1/2}(\partial\Omega_i)}<\infty\}\quad\text{with}\\
	    &\|v\|_{H^{1/2}(\partial\Omega_i)} =
	                 \Big(\int_{\partial\Omega_i}\int_{\partial\Omega_i}\frac{\abs{v(x)-v(y)}^2}{\abs{x-y}^{d}}\,\mathrm{d}x\,\mathrm{d}y
	                        +\|v\|^2_{L^2(\partial\Omega_i)}\Big)^{1/2}.
	\end{aligned}
\end{equation}
Denoting the extension by zero from $\Gamma$ to $\partial\Omega_i$ by $\hat{E}_i$ we define the Lions--Magenes space as
\begin{displaymath}
	    \Lambda=\{\mu\in L^2(\Gamma): \hat{E}_i\mu\in H^{1/2}(\partial\Omega_i)\}\quad\text{with}\quad 
	    \|\mu\|_\Lambda=\|\hat{E}_i\mu\|_{H^{1/2}(\partial\Omega_i)}.
\end{displaymath}
Since $H^{1/2}(\partial\Omega_i)$ is a separable Hilbert space, so is  $\Lambda$. On $V_i$ the trace operator takes the form
\begin{displaymath}
	     \hat{T}_i: V_i\rightarrow \Lambda:v\mapsto\restr{(\hat{T}_{\partial\Omega_i}v)}{\Gamma},
\end{displaymath}	     
and is bounded; see~\cite[Lemma 4.4]{EHEE22}.
\begin{remark}
By~\cite[Lemma A.8]{widlund} we can identify $\Lambda=[H^{1/2}_0(\Gamma), L^2(\Gamma)]_{1/2}$, which shows that $\Lambda$ is independent of $i=1,2$. This reference is without proof, but a proof exists for the case of a smooth $\Omega$ in~\cite[Theorem 11.7]{lionsmagenes1}. The independence of $\Lambda$ on $i=1, 2$ is assumed to be true for the sake of presentation, but it is not a necessary assumption for our convergence analysis; see~\cite{EHEE22}.
\end{remark}

For the temporal function space $H^s(\R)$, $s\in [0,1]$, we use the Fourier definition
\begin{equation}\label{eq:sobolevfourier}
	\begin{aligned}
            &H^s(\R)=\{v\in L^2(\R): (1+(\cdot)^2)^{s/2}\hat{\F} v\in L_\C^2(\R)\}\quad\text{with}\\
             &\|v\|_{H^s(\R)}=\|(1+(\cdot)^2)^{s/2}\hat{\F} v\|_{L_\C^2(\R)}.
         \end{aligned}
\end{equation}
Here, $\hat{\F}$ is the Fourier transform and $L_\C^2(\R)$ is the complexification of the real Hilbert space $L^2(\R)$. Note that this is equivalent to the Sobolev--Slobodetskii definition \cref{eq:slobodetskii} for $s=1/2$ and $\partial\Omega_{i}=\R$; see~\cite[Lemma 16.3]{tartar}. We then introduce the temporal Hilbert transform by
\begin{displaymath}
        \hat{\Hc} v(t)=\lim_{\epsilon\rightarrow 0^{+}}\frac{1}{\pi}\int_{\abs{s}\geq\epsilon}\frac{1}{s}v(t-s)\,\mathrm{d}s.
\end{displaymath}
\noindent
From \cite[Chapters 4, 5]{king} we have that $\hat{\Hc}: L^2(\R)\rightarrow L^2(\R)$ is an isomorphism with inverse $\hat{\Hc}^{-1}=-\hat{\Hc}$ and
\begin{equation}\label{eq:hilbertfourier}
        \hat{\Hc}=\hat{\F}^{-1}\hat{M}_\sgn\hat{\F},
\end{equation}
where $\hat{M}_\sgn v(\xi)=-\mathrm{i}\,\sgn(\xi)v(\xi)$. The formula \cref{eq:hilbertfourier} combined with the definition \cref{eq:sobolevfourier} shows that $\hat{\Hc}: H^s(\R)\rightarrow H^s(\R)$ is an isomorphism. We also introduce the temporal half-derivatives as
\begin{displaymath}
\hat{\partial}^{1/2}_\pm=\hat{\F}^{-1}\hat{M}_{\pm}\hat{\F},
\end{displaymath}
where $\hat{M}_+v(\xi)=\sqrt{\mathrm{i}\xi} v(\xi)$ and $\hat{M}_-v(\xi)=\overline{\sqrt{\mathrm{i}\xi}}v(\xi)$. It is clear from the definition \cref{eq:sobolevfourier} that $\hat{\partial}^{1/2}_\pm: H^{1/2}(\R)\rightarrow L^2(\R)$ are bounded linear operators. The important relations between these operators are given in the lemma below.
\begin{lemma}\label{lemma:timeder}
For $v\in H^{1/2}(\R)$ one has the equalities 
\begin{displaymath}
\hat{\partial}^{1/2}_+ v = -\hat{\partial}^{1/2}_-\hat{\Hc} v\quad\text{and}\quad (\hat{\partial}^{1/2}_+ v\,,\, \hat{\partial}^{1/2}_- v)_{L^2(\R)}=0.
\end{displaymath}
Moreover, for $v\in H^1(\R)$ and $w\in H^{1/2}(\R)$ one has the fractional integration by parts formula
\begin{displaymath}
(\hat{\partial}_t v\,,\, w)_{L^2(\R)}=(\hat{\partial}^{1/2}_+ v\,,\, \hat{\partial}^{1/2}_- w)_{L^2(\R)}.
\end{displaymath}
\end{lemma}
\begin{proof}
Let $v\in H^{1/2}(\R)$ and observe that
\begin{displaymath}
\sqrt{\mathrm{i}\xi}=-\overline{\sqrt{\mathrm{i}\xi}}(-\mathrm{i}\,\sgn(\xi)). 
\end{displaymath}
The Fourier characterization of the operators $\hat{\Hc},\hat{\partial}^{1/2}_+, \hat{\partial}^{1/2}_-$ then implies that 
\begin{displaymath}
\hat{\partial}^{1/2}_+v =-\hat{\partial}^{1/2}_-\hat{\Hc}v. 
\end{displaymath}
A similar argument, together with the fact that $|\hat{\F}v|^2$ is an even function, shows that
\begin{displaymath}
(\hat{\partial}^{1/2}_+ v, \hat{\partial}^{1/2}_- v)_{L^2(\R)}=\int_\R i\xi|\hat{\F}v|^2\mathrm{d}\xi=0
\end{displaymath}
for $v\in H^{1/2}(\R)$. Finally, the fractional integration by parts formula follows by the Fourier characterization $\hat{\partial}_t=\hat{\F}^{-1}\hat{M}\hat{\F}$, where $\hat{M}v(\xi)=\mathrm{i} \xi v(\xi)$. 
\qed
\end{proof}
The Fourier characterization of the operator $\hat{\partial}^{1/2}_+$ and~\cref{eq:sobolevfourier} yield that
\begin{displaymath}
v\mapsto\bigl(\|\hat{\partial}^{1/2}_+v\|^2_{L^2(\R)}+\|v\|^2_{L^2(\R)}\bigr)^{1/2}
\end{displaymath}
is an equivalent norm on $H^{1/2}(\R)$.

\section{Tensor spaces}\label{sec:tensor}

Inspired by the finite element analysis in~\cite{steinbach20}, we identify our Bochner spaces in space-time as tensor spaces. A general introduction to tensor spaces can be found in ~\cite[Chapter 3.4]{weidmann}. 

We denote the algebraic tensor product of two (real) separable Hilbert spaces $X, Y$ by $X\otimes Y$. For elements of the form $x\otimes y$ the inner product is defined as
\begin{displaymath}
(x_1\otimes y_1, x_2\otimes y_2)_{X\tensor Y}=(x_1, x_2)_X(y_1, y_2)_Y
\end{displaymath}
and for arbitrary elements in $X\otimes Y$ the definition is extended by linearity. The closure of $X\otimes Y$ with respect to this norm is denoted by $X\tensor Y$. From these definitions it follows that  $X\tensor Y = Y\tensor X$. 

If $\{x_k\}_{k\geq 0}$ and $\{y_\ell\}_{\ell\geq 0}$ are orthonormal bases of $X$ and $Y$, respectively, then\linebreak $\{x_k\otimes y_\ell\}_{k,\ell\geq 0}$ is an orthonormal basis of $X\tensor Y$ and every $v\in X\tensor Y$ can be represented as 
\begin{equation}\label{eq:uiorth}
v=\sum_{k,\ell=0}^\infty c_{k,\ell}(x_k \otimes y_\ell)=\sum_{k=0}^\infty x_k \otimes z_k ,\quad z_k=\sum_{\ell=0}^\infty c_{k,\ell}y_\ell\in Y,
\end{equation} 
where $\{c_{k,\ell}\}_{k,\ell\geq 0}$ are real coefficients. This follows from~\cite[Theorem 3.12]{weidmann}. 

We recall the following result on the extensions of operators to tensor spaces. The proof of \cref{lemma:tensor} can be found in~\cite[Section 12.4.1]{aubin}. 
\begin{lemma}\label{lemma:tensor}
Let $X_{k}, Y_{k}$, $k=1,2$, be separable Hilbert spaces, and $A: X_1\rightarrow X_2$, $B: Y_1\rightarrow Y_2$ be bounded linear operators. Then there is a bounded linear operator
\begin{displaymath}
                A\tensor B: X_1\tensor Y_1\rightarrow X_2\tensor Y_2
\end{displaymath}
 such that $(A\tensor B)(x\otimes y)=Ax\otimes By$ for every $x\in X_1, y\in Y_1$.
\end{lemma}
From \cref{lemma:tensor} it follows that the parabolic trace operators
\begin{align*}
            T_{\partial\Omega_i}&=I\tensor\hat{T}_{\partial\Omega_i}:L^2(\R)\tensor H^1(\Omega)\rightarrow L^2(\R)\tensor H^{1/2}(\partial\Omega_i), \\
            T_i&=I\tensor\hat{T}_i:L^2(\R)\tensor V_i\rightarrow L^2(\R)\tensor \Lambda
\end{align*}
are bounded. Furthermore, we have the identity
\begin{displaymath}
T_iv=\restr{(T_{\partial\Omega_i}v)}{\Gamma\times\R}\quad\text{for }v\in L^2(\R)\tensor V_i. 
\end{displaymath}
This is easily proven by validating the identity on the dense subset $L^2(\R)\otimes V_i$ and using the continuity of the operators. Note that the restrictions, as well as the extension by zero $E_i=I\tensor\hat{E}_i$, are all well defined operations due to~\cref{lemma:tensor}. 

For any separable Hilbert space $X$ one has the identification
\begin{displaymath}
H^s(\R)\tensor X\cong H^s(\R, X),
\end{displaymath}
where $H^s(\R, X)$, $s\in [0,1]$, is a Sobolev--Bochner space; see~\cite[Chapter 2.5.d]{hytonen}. This can be proven by noting that the norms coincide on $H^s(\R)\otimes X$, the identification $H^1(\R)\tensor X\cong H^1(\R, X)$; see~\cite[Theorem 12.7.1]{aubin}, and the fact that $H^1(\R)\otimes X$ and $H^1(\R, X)$ are dense in $H^s(\R)\tensor X$ and $H^s(\R, X)$, respectively; see~\cite[Theorem 3.12]{weidmann} and~\cite[Proposition 6.1]{lionsmagenes1}.

The Sobolev--Bochner spaces below will make up the core of the analysis:
\begin{align*}
            W&=H^{1/2}(\R)\tensor L^2(\Omega)\,\cap\, L^2(\R)\tensor V, \\
	    W_i^0&=H^{1/2}(\R)\tensor L^2(\Omega_i)\,\cap\, L^2(\R)\tensor V_i^0, \\
	    W_i&=H^{1/2}(\R)\tensor L^2(\Omega_i)\,\cap\, L^2(\R)\tensor V_i, \\
	     Z&=H^{1/4}(\R)\tensor L^2(\Gamma)\,\cap\, L^2(\R)\tensor \Lambda.
\end{align*}
\begin{lemma}\label{lemma:identities}
Suppose \cref{ass:domains} holds. Then we have the identities
\begin{align}
            L^2(\R)\tensor V_i^0&=\{v\in L^2(\R)\tensor H^1(\Omega_i): T_{\partial\Omega_i}v=0\},\label{eq:Wi0}\\
            L^2(\R)\tensor V_i&=\{v\in L^2(\R)\tensor H^1(\Omega_i): \restr{(T_{\partial\Omega_i}v)}{(\partial\Omega_i\setminus\Gamma)\times\R}=0\},\label{eq:Wi}\\
            L^2(\R)\tensor \Lambda&=\{\mu\in L^2(\R)\tensor L^2(\Gamma): E_i\mu\in L^2(\R)\tensor H^{1/2}(\partial\Omega_i)\}.\label{eq:Z}
\end{align}
\end{lemma}

\begin{proof}
Throughout the proof let $\{x_k\}_{k\geq 0}$ be an orthonormal basis of $L^2(\R)$. Let $\{y_\ell\}_{\ell\geq 0}$ be an orthonormal basis of a separable Hilbert space $Y$ given by the context, and the corresponding $\{z_k\}_{k\geq 0}$ are defined as in~\cref{eq:uiorth}.

To prove the identity~\cref{eq:Wi0} recall that $V_i^0=\{y\in H^1(\Omega_i): \hat{T}_{\partial\Omega_i}y=0\}$; see~\cite[Theorem~6.6.4]{kufner}. 
This observation together with the continuity of $T_{\partial\Omega_i}$ then yields that every $v\in L^2(\R)\tensor V_i^0$ satisfies 
\begin{displaymath}
T_{\partial\Omega_i}v=\sum_{k,\ell=0}^\infty c_{k,\ell}(x_k \otimes \hat{T}_{\partial\Omega_i}y_\ell)=0.
\end{displaymath}
Conversely, suppose that $v\in L^2(\R)\tensor H^1(\Omega_i)$ with $T_{\partial\Omega_i}v=0$. By the continuity of $T_{\partial\Omega_i}$ we have 
\begin{displaymath}
0=T_{\partial\Omega_i}v=\sum_{k=0}^\infty x_k \otimes \hat{T}_{\partial\Omega_i}z_k.
\end{displaymath}
It then follows by the orthonormality of $\{x_k\}_{k\geq 0}$ that 
\begin{equation}\label{eq:ort}
0 = (T_{\partial\Omega_i}v \, ,\, x_k\otimes \hat{T}_{\partial\Omega_i}z_k)_{L^2(\R)\tensor L^{2}(\partial\Omega_i)}=\|\hat{T}_{\partial\Omega_i}z_k\|^2_{L^2(\partial\Omega_i)}.
\end{equation}
Hence, $z_k\in V_i^0$, $k=0,1,\dots$, and from~\cref{eq:uiorth} we can conclude that $v\in L^2(\R)\tensor V_i^0$.

The proof of~\cref{eq:Wi} is similar after observing the definition $V_i=\{v\in H^1(\Omega_i): \restr{(\hat{T}_{\partial\Omega_i}v)}{\partial\Omega_i\setminus\Gamma}=0\}$ and employing the $L^2(\R)\tensor L^{2}(\partial\Omega_i\setminus\Gamma)$ inner product in~\cref{eq:ort} instead of $L^2(\R)\tensor L^{2}(\partial\Omega_i)$.

To prove~\cref{eq:Z} first observe that $E_i$ can be interpreted as a map from $L^2(\R)\tensor \Lambda$ into $L^2(\R)\tensor H^{1/2}(\partial\Omega_i)$, by~\cref{lemma:tensor}. This immediately gives the inclusion from left to right. Conversely, assume that $\mu\in L^2(\R)\tensor L^2(\Gamma)$ with $E_i\mu\in L^2(\R)\tensor H^{1/2}(\partial\Omega_i)$. By~\cref{eq:uiorth}, we have the representation
\begin{displaymath}
E_i\mu=\sum_{k=0}^\infty x_k \otimes z_k,\quad z_k\in H^{1/2}(\partial\Omega_i).
\end{displaymath}
As $\restr{(E_i\mu)}{(\partial\Omega_i\setminus\Gamma)\times\R}=0$, the orthonormality of $\{x_k\}_{k\geq 0}$ yields that
\begin{displaymath}
0 = (E_i\mu \, ,\, x_k\otimes z_k)_{L^2(\R)\tensor L^{2}(\partial\Omega_i\setminus\Gamma)}=\|z_k\|^2_{L^2(\partial\Omega_i\setminus\Gamma)},
\end{displaymath}
i.e., $\restr{z_k}{\partial\Omega_i\setminus\Gamma}=0$. As $\hat{E}_i$ is an isometry from $\Lambda$ onto $\{\mu\in H^{1/2}(\partial\Omega_i): \restr{\mu}{\partial\Omega_i\setminus\Gamma}=0\}$; compare with~\cite[Lemma~4.1]{EHEE22}, we obtain that $\restr{z_k}{\Gamma}\in\Lambda$ for all $k=0,1,\dots$ Since
\begin{displaymath}
\mu=\restr{(E_i\mu)}{\Gamma\times\R}=\sum_{k=0}^\infty x_k \otimes \restr{z_k}{\Gamma}\quad\text{in }L^2(\R)\tensor L^2(\Gamma)
\end{displaymath}
and $\{\sum_{k=0}^n x_k \otimes \restr{z_k}{\Gamma}\}_{n\geq0}$ is a Cauchy sequence in $L^2(\R)\tensor \Lambda$, we conclude that\linebreak $\mu\in L^2(\R)\tensor \Lambda$.
\qed
\end{proof}

\begin{lemma}\label{lemma:density}
If \cref{ass:domains} holds, then $Z$ is dense in $L^2(\Gamma\times\R)$.
\end{lemma}
\begin{proof}
The spaces $H^{1/4}(\R)$ and $\Lambda$ are dense in $L^2(\R)$ and $L^2(\Gamma)$, respectively; see~\cite[Lemma 15.10]{tartar} and \cite[Lemma 4.2]{EHEE22}. Therefore, by \cite[Theorem 3.12]{weidmann}, the corresponding algebraic tensor space $H^{1/4}(\R)\otimes \Lambda$ is dense in $L^2(\R)\tensor L^2(\Gamma)\cong L^2(\Gamma\times\R)$. The density of $Z$ in $L^2(\Gamma\times\R)$ then follows as $H^{1/4}(\R)\otimes\Lambda\subset Z$.
\qed
\end{proof}
\begin{lemma}\label{lemma:tracei}
If \cref{ass:domains} holds, then  $T_i:W_i\rightarrow Z$ is bounded.
\end{lemma}
\begin{proof}
It follows from the density of $C^\infty_0(\R)\otimes C^\infty(\Bar{\Omega}_i)$ in $L^2(\R)\tensor H^1(\Omega_i)$; see \cite[Theorem 3.12]{weidmann}, that our definition of $T_{\partial\Omega_i}$ coincide with the definition given in~\cite[Lemma 2.4]{costabel90}. Hence, the restricted trace operator
\begin{displaymath}
 T_{\partial\Omega_i}:H^{1/2}(\R)\tensor L^2(\Omega_i)\cap L^2(\R)\tensor H^1(\Omega_i)\rightarrow
                                   H^{1/4}(\R)\tensor L^2(\Omega_i)\cap L^2(\R)\tensor H^{1/2}(\partial\Omega_i)
\end{displaymath}
is well defined and bounded. 

If $v\in W_i$ then $T_iv=\restr{(T_{\partial\Omega_i}v)}{\Gamma\times\R}\in H^{1/4}(\R)\tensor L^2(\Gamma)$. As $W_i\subset L^2(\R)\tensor V_i $, we have by definition that $T_iv\in L^2(\R)\tensor \Lambda$, i.e., $T_i:W_i\rightarrow Z$. The boundedness of $T_i:W_i\rightarrow Z$ then follows as the operators $T_{\partial\Omega_i}:H^{1/2}(\R)\tensor L^2(\Omega_i)\cap L^2(\R)\tensor H^1(\Omega_i)\rightarrow H^{1/4}(\R)\tensor L^2(\Omega_i)$ and $T_i:L^2(\R)\tensor V_i\rightarrow L^2(\R)\tensor \Lambda$ are bounded.
\qed
\end{proof}

\begin{lemma}\label{lemma:abstracthilbert}
Let $X$ be a separable Hilbert space and $s\in [0,1]$. The Hilbert transform 
\begin{displaymath}
\Hc=\hat{\Hc}\tensor I:H^s(\R)\tensor X\rightarrow H^s(\R)\tensor X
\end{displaymath} 
is then an isomorphism.
\end{lemma}
\begin{proof}
According to \cref{lemma:tensor} the operator $\hat{\Hc}\otimes I$ extends to a bounded linear operator $\Hc: H^s(\R)\tensor X\rightarrow H^s(\R)\tensor X$. From \cref{eq:hilbertfourier} it follows that $\Hc:H^s(\R)\tensor X\rightarrow H^s(\R)\tensor X$ is an isomorphism.
\qed
\end{proof}

\begin{lemma}\label{lemma:hilbert}
If \cref{ass:domains} holds then the restricted operators
\begin{displaymath} 
\Hc_i:W_i\rightarrow W_i\quad\text{and}\quad\Hc_\Gamma:Z\rightarrow Z
\end{displaymath} 
are isomorphisms and satisfy $T_i\Hc_iv=\Hc_\Gamma T_iv$ for all $v\in W_i$.
\end{lemma}
\begin{proof}
By different choices of $s\in [0,1]$ and separable Hilbert spaces $X$ in~\cref{lemma:abstracthilbert} one obtains the isomorphisms
\begin{align*}
        \Hc_i&:L^2(\R)\tensor L^2(\Omega_i)\rightarrow L^2(\R)\tensor L^2(\Omega_i)\\
        \Hc_i&:H^{1/2}(\R)\tensor L^2(\Omega_i)\rightarrow H^{1/2}(\R)\tensor L^2(\Omega_i)\\
        \Hc_i&:L^2(\R)\tensor V_i\rightarrow L^2(\R)\tensor V_i.
\end{align*}
Since the operators coincide on $W_i$ we have that $\Hc_i:W_i\rightarrow W_i$ is an isomorphism. A similar argument shows that $\Hc_\Gamma:Z\rightarrow Z$ is an isomorphism. For $v\in L^2(\R)\otimes V_i$ we have
\begin{displaymath} 
T_i\Hc_iv=(I\otimes \hat{T}_i)(\hat{\Hc}\otimes I_{L^2(\Omega_i)})v=
                \hat{\Hc}\otimes \hat{T}_iv=
                (\hat{\Hc}\otimes I_{L^2(\Gamma)})(I\otimes \hat{T}_i)v=
                \Hc_\Gamma T_iv,
\end{displaymath} 
and by density the identity holds for $v\in W_i\subset L^2(\R)\tensor V_i$.
\qed
\end{proof}
Finally, let $\varphi\in [0,\pi/2]$ be a parameter to be chosen later and define
\begin{displaymath} 
\Hc_i^\varphi=\cos{(\varphi)}I-\sin{(\varphi)}\Hc_i\quad\text{and}\quad\Hc_\Gamma^\varphi=\cos{(\varphi)}I-\sin{(\varphi)}\Hc_\Gamma.
\end{displaymath} 
It follows from \cref{lemma:hilbert} that $\Hc_i^\varphi:W_i\rightarrow W_i$ and $\Hc_\Gamma^\varphi:Z\rightarrow Z$ are isomorphisms and
\begin{equation}\label{eq:hilbertphi}
T_i\Hc_i^\varphi=\Hc_\Gamma^\varphi T_i.
\end{equation}

\section{Variational space-time formulations}\label{sec:weak}
By \cref{lemma:tensor} the extensions of the spatial gradient
\begin{displaymath} 
\nabla: L^2(\R)\tensor H^1(\Omega_i)\rightarrow L^2(\R)\tensor L^2(\Omega_i)^d
\end{displaymath} 
and the temporal half-derivatives
\begin{displaymath} 
\partial^{1/2}_\pm: H^{1/2}(\R)\tensor L^2(\Omega_i)\rightarrow L^2(\R)\tensor L^2(\Omega_i)
\end{displaymath} 
are all bounded linear operators. Note that we leave out the dependence on $i=1,2$ on the above operators for the sake of readability.

The bilinear forms $a:W\times W\rightarrow \R$ and $a_i: W_i\times W_i\rightarrow\R$, $i=1,2$, are defined by the formulas
\begin{align*}
        a(u, v)&=\int_\R\int_{\Omega}\partial^{1/2}_+u\partial^{1/2}_-v+\nabla u\cdot\nabla v\,\mathrm{d}x\,\mathrm{d}t\quad\text{and}\\
        a_i(u_i, v_i)&=\int_\R\int_{\Omega_i}\partial^{1/2}_+u_i\partial^{1/2}_-v_i+\nabla u_i\cdot\nabla v_i\,\mathrm{d}x\,\mathrm{d}t,
\end{align*}
respectively. For $f\in W^*$ the weak, or variational, formulation of the heat equation~\cref{eq:strong} is to find $u\in W$ such that
\begin{equation}\label{eq:weak}
	    a(u, v)=\langle f, v\rangle\quad \textrm{ for all }v\in W.
\end{equation}
Similarly, for $f\in (W_i^0)^*$ the weak problems on $\Omega_i\times\R$ is to find $u_i\in W_i^0$ such that
\begin{equation}\label{eq:weakai}
a_i(u_i, v_i)=\langle f_i, v_i\rangle\quad \textrm{ for all } v_i\in W_i^0.
\end{equation}
Here, $X^*$ denotes the dual of the Hilbert space $X$ and $\langle\cdot, \cdot\rangle$ denotes the dual pairing in $X^*\times X$. The weak problem \cref{eq:weak} can be derived by multiplying \cref{eq:strong} by $v\in W$, integrating and using the fractional integration by parts formula from \cref{lemma:timeder} extended to the tensor setting. The weak formulations \cref{eq:weakai} follow in the same way.

A bilinear form $a:X\times X\rightarrow \R$ is referred to as coercive in $Y$, where $X\subseteq Y$, if
\begin{displaymath} 
        a(v, v)\geq c\|v\|_Y^2\quad \textrm{for all } v\in X,
\end{displaymath} 
and coercive-equivalent in $X$ with the isomorphism $A:X\rightarrow X$ if
\begin{displaymath} 
        a(v, Av)\geq c\|v\|_X^2\quad \textrm{for all } v\in X.
\end{displaymath} 
The following result follows the argumentation in~\cite[Section 2.8]{larsson15}, which is a summary of the coercive-equivalency idea from~\cite{fontesthesis}. 
\begin{lemma}\label{lemma:ai}
Let~\cref{ass:domains} be valid. The bilinear forms $a_i$ are then bounded in $W_i$, coercive in $L^2(\R)\tensor H^1(\Omega_i)$, and coercive-equivalent in $W_i$ with the isomorphisms $\Hc_i^\varphi$ for a sufficiently small $\varphi>0$. An analogous result holds for the bilinear form~$a$. In particular there exists unique solutions to \cref{eq:weak,eq:weakai}, respectively.
\end{lemma}
\begin{proof}
We will only consider the case with the bilinear form~$a_i$, as the same proof holds for~$a$. Here, the boundedness of $a_i$ on $W_i$ follows directly by the Cauchy--Schwarz inequality together with the boundedness of the operators $\nabla$ and $\partial^{1/2}_\pm$.

To prove the coercivity properties first observe that the identities in~\cref{lemma:timeder} can trivially be validated on $H^{1/2}(\R)\otimes L^2(\Omega_i)$ and extended to $H^{1/2}(\R)\tensor L^2(\Omega_i)$ by density and the boundedness of the related operators. That is, 
\begin{displaymath}
\partial^{1/2}_+ v = -\partial^{1/2}_-\Hc_i v\quad\text{and}\quad(\partial^{1/2}_+ v\, ,\, \partial^{1/2}_- v)_{L^2(\R)\tensor L^2(\Omega_i)}=0
\end{displaymath} 
for every $v\in H^{1/2}(\R)\tensor L^2(\Omega_i)$. Second, as~\cref{ass:domains} is valid, a similar density argument together with the observations regarding the equivalent norms in~\cref{sec:prel} yields the ``extended'' Poincaré's inequality
\begin{displaymath}
\|\nabla v\|_{L^2(\R)\tensor L^2(\Omega_i)^d}\geq c\|v\|_{L^2(\R)\tensor L^2(\Omega_i)}\quad\text{for all }v\in L^2(\R)\tensor H^1(\Omega_i),
\end{displaymath}
and the fact that 
\begin{displaymath}
v\mapsto\bigl(\|\partial^{1/2}_+v\|^2_{H^{1/2}(\R)\tensor L^2(\Omega_i)}+\|\nabla v\|^2_{L^2(\R)\tensor L^2(\Omega_i)^d}\bigr)^{1/2}
\end{displaymath}
is an equivalent norm on $W_i$. Third, by~\cref{lemma:abstracthilbert}, $\nabla\Hc_i$ is a bounded operator from $L^2(\R)\tensor H^1(\Omega_i)$ into $L^2(\R)\tensor L^2(\Omega_i)^d$. With these observations we have the bound
\begin{align*}
a_i(v, \Hc_i^\varphi v)&=\cos(\varphi)(\partial^{1/2}_+ v\, ,\, \partial^{1/2}_- v)_{L^2(\R)\tensor L^2(\Omega_i)}\\
                                  &\quad+\sin(\varphi)(\partial^{1/2}_+ v\, ,\, -\partial^{1/2}_-\Hc_i v)_{L^2(\R)\tensor L^2(\Omega_i)}\\
                                  &\quad+\cos(\varphi)\|\nabla v\|^2_{L^2(\R)\tensor L^2(\Omega_i)^d}
                                     -\sin(\varphi)(\nabla v\, ,\, \nabla \Hc_i v)_{L^2(\R)\tensor L^2(\Omega_i)^d}\\
                                  &\geq\sin(\varphi)\|\partial^{1/2}_+v\|^2_{L^2(\R)\tensor L^2(\Omega_i)}
                                     +\bigl(\cos(\varphi)-c\sin(\varphi)\bigr)\|\nabla v\|^2_{L^2(\R)\tensor L^2(\Omega_i)^d}
 \end{align*}
for every $v\in W_i$. Thus, by choosing a sufficiently small $\varphi>0$ we have the sought after coercive-equivalency
\begin{displaymath}
            a_i(v, \Hc_i^\varphi v)\geq c\|v\|_{W_i}^2\quad\text{for every }v\in W_i,
\end{displaymath}
and choosing $\varphi=0$ yields the coercivity bound 
\begin{displaymath}
a_i(v,v)\geq c\|v\|_{L^2(\R)\tensor H^1(\R)}^2\quad\text{for every }v\in W_i.
\end{displaymath}

The derived properties of $a_i$ yield the existence and uniqueness of a solution for~\cref{eq:weakai}, which follows by~\cite[Corollary~2.2]{larsson15}. \qed
\end{proof}

The next lemma is required in order to consider non-homogeneous boundary values on the interface. 
\begin{lemma}\label{lemma:Fi}
Suppose that \cref{ass:domains} holds. For any $\eta\in Z$ there exists $u_i\in W_i$ such that $T_iu_i=\eta$ and
\begin{displaymath}
a_i(u_i, v_i)=0\quad\text{for all }v_i\in W_i^0.
\end{displaymath}
Moreover, $\|u_i\|_{W_i}\leq C\|\eta\|_Z$.
\end{lemma}
\begin{proof}
From \cref{eq:Z} we have that $E_i\eta\in H^{1/4}(\R)\tensor L^2(\partial\Omega_i)\cap L^2(\R)\tensor H^{1/2}(\partial\Omega_i)$.
By \cite[Theorem 2.9 \& Remark~2.10]{costabel90}, there exists a weak solution 
\begin{displaymath}
u_i\in H^{1/2}(\R)\tensor L^2(\Omega_i)\cap L^2(\R)\tensor H^1(\Omega_i)
\end{displaymath}
such that $T_{\partial\Omega_i}u_i=E_i\eta$. It follows from \cref{eq:Wi} that $u_i\in W_i$. The bound $\|u_i\|_{W_i}\leq C\|\eta\|_Z$ follows by~\cite[p.~515]{costabel90} together with the definitions of the related norms. \qed
\end{proof}
According to \cref{lemma:Fi} there exists a bounded linear operator $F_i:Z\rightarrow W_i:\eta\mapsto u_i$ such that
\begin{equation}\label{eq:Fi}
	    a(F_i\eta, v_i)=0\quad \textrm{ for all }v_i\in W_i^0,
\end{equation}
with left inverse $T_i$, i.e., $T_iF_i\eta=\eta$ for $\eta\in Z$. From \cref{lemma:ai} there also exists a bounded linear operator $G_i:(W_i^0)^*\rightarrow W_i^0:f_i\mapsto u_i$ such that
\begin{equation}\label{eq:Gi}
	    a_i(G_if_i, v_i)=\langle f_i, v_i\rangle\quad \textrm{for all }v_i\in W_i^0.
\end{equation}
These two operators are the main ingredients for stating the transmission problem and then reformulating the problem on the space-time interface $\Gamma\times\R$. Before moving on to the transmission problem it is necessary to prove that we can ``paste'' together functions in our $H^{1/2}$-framework. This is the purpose of the following lemma.

\begin{lemma}\label{lemma:cutpaste}
Suppose that \cref{ass:domains} holds. Let $(v_1,v_2)\in W_1\times W_2$ and define $v=\{v_i\textrm{ on }\Omega_i\times\R,\, i=1,2\}$. If $T_1v_1=T_2v_2$ then $v\in W$. Conversely, let $v\in W$ and define $v_i=\restr{v}{\Omega_i\times\R}$. Then $T_iv_i\in Z$ and $T_1v_1=T_2v_2$.
\end{lemma}
\begin{proof}
Assume that $v_i\in W_i$, $i=1,2$, and $T_1v_1=T_2v_2$, then 
\begin{displaymath}
v=\{v_i\textrm{ on }\Omega_i\times\R,\, i=1,2\}\in L^2(\R)\tensor L^2(\Omega).
\end{displaymath}
In order to prove the $L^2(\R)\tensor V$-regularity of $v$, let $\{x_k\}_{k\geq 0}$ be an orthonormal basis of $L^2(\R)$ and $\{(y_i)_\ell\}_{\ell\geq 0}$ be orthonormal bases of $V_i$, $i=1,2$. The corresponding elements $\{(z_i)_k\}_{k\geq 0}\subset V_i$ are defined as in~\cref{eq:uiorth}. This yields the representation
\begin{displaymath}
v_i=\sum_{k=0}^\infty x_k \otimes (z_i)_k
\end{displaymath}
and the equality 
\begin{displaymath}
 0=(T_1v_1-T_2v_2\, ,\,x_k\otimes (\hat{T}(z_1)_k-\hat{T}_2(z_2)_k)_{L^2(\R)\tensor L^2(\Gamma)}=\|\hat{T}_1(z_1)_k-\hat{T}_2(z_2)_k\|^2_{ L^2(\Gamma)}.
\end{displaymath}
That is,  $\hat{T}_1(z_1)_k=\hat{T}_2(z_2)_k$, for all $k=0,1,\ldots$, and from \cite[Lemma 4.6]{EHEE22} it follows that $z_k=\{(z_i)_k \textrm{ on }\Omega_i\times\R,\, i=1,2\}\in V$. We also have the identification
\begin{displaymath}
v=\sum_{k=0}^\infty x_k \otimes z_k\quad\text{in }L^2(\R)\tensor L^2(\Omega).
\end{displaymath}
As $\|z_k\|^2_{V}=\|(z_1)_k\|^2_{V_1}+\|(z_2)_k\|^2_{V_2}$, one has that $\{\sum_{k=0}^n x_k \otimes z_k\}_{n\geq0}$ is a Cauchy sequence in $L^2(\R)\tensor V$. Hence, $v$ is also an element in $L^2(\R)\tensor V$. 

The $H^{1/2}(\R)\tensor L^2(\Omega)$-regularity of $v$ follows in a similar fashion, by expanding the $v_i$ elements in terms of an orthonormal basis $\{x_k\}_{k\geq 0}$ of $H^{1/2}(\R)$ and  an orthonormal basis $\{(y_i)_k\}_{k\geq 0}$ of $L^2(\Omega_i)$ together with the observation that $\|z_k\|^2_{L^2(\Omega)}=\|(z_1)_k\|^2_{L^2(\Omega_1)}+\|(z_2)_k\|^2_{L^2(\Omega_2)}$. In conclusion, $v\in W$.

Conversely, let $v\in W$. From~\cref{eq:Wi} we get that $v_i=\restr{v}{\Omega_i\times\R}$ is an element in $W_i$ and, by~\cref{lemma:tracei}, $T_iv_i\in Z$. Let $\{x_k\}_{k\geq 0}$ be an orthonormal basis of $L^2(\R)$ and $\{y_\ell\}_{\ell\geq 0}$ an orthonormal basis of $V$. The related elements $\{z_k\}_{k\geq 0}\subset V$ are given by~\cref{eq:uiorth} and we obtain
\begin{displaymath}
v=\sum_{k=0}^\infty x_k \otimes z_k\quad\text{and}\quad v_i=\sum_{k=0}^\infty x_k \otimes (z_i)_k,
\end{displaymath}
where $(z_i)_k=\restr{z_k}{\Omega_i\times\R}\in V_i$. By~\cite[Lemma 4.5]{EHEE22} it follows that $\hat{T}_1(z_1)_k=\hat{T}_2(z_2)_k$ for every $k=0,1,\ldots$ Therefore
\begin{displaymath}
T_1v_1=\sum_{k,\ell=0}^\infty x_k \otimes \hat{T}_1(z_1)_k=\sum_{k,\ell=0}^\infty x_k \otimes \hat{T}_2(z_2)_k=T_2v_2,
\end{displaymath}
and the sought after equality is obtained. \qed
\end{proof}

The existence and uniqueness theory holds for any $f\in W^*$, but for the convergence analysis we need a stricter requirement. For simplicity we assume the following:
\begin{assumption}\label{ass:f}
The source term $f$ is an element in $L^2(\Omega\times\R)$.
\end{assumption}
With \cref{ass:f} we can define $f_i=\restr{f}{\Omega_i\times\R}$, which we will assume to hold for the remainder of the analysis. The weak transmission problem is then to find $(u_1, u_2)\in W_1\times W_2$ such that
\begin{equation}\label{eq:weaktran}
	\left\{\begin{aligned}
	     a_i(u_i, v_i)&=(f_i, v_i)_{L^2(\Omega_i)} & & \text{for all } v_i\in W_i^0,\, i=1,2,\\
	     T_1u_1&=T_2u_2, & &\\
	     \textstyle\sum_{i=1}^2 a_i(& u_i, F_i\mu)-(f_i, F_i\mu)_{L^2(\Omega_i)}=0 & &\text{for all }\mu\in Z. 
	\end{aligned}\right.
\end{equation}
\begin{lemma}\label{lemma:tranequiv}
Suppose that \cref{ass:domains,ass:f} hold. Then the weak heat equation is equivalent to the weak transmission problem in the following way: If $u$ solves \cref{eq:weak} then $(u_1, u_2)=(\restr{u}{ \Omega_1\times\R}, \restr{u}{ \Omega_2\times\R})$ solves \cref{eq:weaktran}. Conversely, if $(u_1, u_2)$ solves \cref{eq:weaktran} then $u=\{u_i \textrm{ on } \Omega_i\times\R\, , i=1,2\}$ solves \cref{eq:weak}. In particular, there exists a unique solution to \cref{eq:weaktran}.
\end{lemma}
\begin{remark}
The proof of \cref{lemma:tranequiv} follows by the same argument as~\cite[Lemma 1.2.1]{quarteroni} and requires that \cref{lemma:cutpaste} holds. This is one of the reason that the analysis is performed in the $H^{1/2}$-setting. As already stated in the introduction, analogous results to \cref{lemma:cutpaste} are not always true; see for instance~\cite[Example 2.14]{costabel90} for a counterexample in $H^1(\R)\tensor H^{-1}(\Omega)\cap L^2(\R)\tensor H^1(\Omega)$.
\end{remark}

The weak form of the space-time Robin--Robin method is the following: Let $u_2^0\in W_2$ be an initial guess and $s>0$ be the fixed method parameter. For each $n=1,2,\ldots$ find functions $(u_1^n, u_2^n)\in W_1\times W_2$ such that
\begin{equation}\label{eq:robin}
    \left\{\begin{aligned}
    	&a_1(u_1^{n+1}, v_1)=(f_1, v_1)_{L^2( \Omega_1\times\R)} & &\text{for all } v_1\in W_1^0,\\
         &a_1(u_1^{n+1}, F_1\mu)-(f_1, F_1\mu)_{L^2( \Omega_1\times\R)}+a_2(u_2^n, F_2\mu)  & & \\
    			&\quad-(f_2, F_2\mu)_{L^2( \Omega_2\times\R)}=s(T_{2}u_2^n-T_{1}u_1^{n+1}, \mu)_{L^2(\Gamma\times\R)} & &\text{for all }\mu\in Z,\\[5pt]
    	&a_2(u_2^{n+1}, v_2)=(f_2, v_2)_{L^2( \Omega_2\times\R)} & &\text{for all } v_2\in W_2^0,\\
    	&a_2(u_2^{n+1}, F_2\mu)-(f_2, F_2\mu)_{L^2( \Omega_2\times\R)}+a_1(u_1^{n+1}, F_1\mu) & &\\
    			&\quad-(f_1, F_1\mu)_{L^2( \Omega_1\times\R)}=s(T_{1}u_1^{n+1}-T_{2}u_2^{n+1}, \mu)_{L^2(\Gamma\times\R)} & &\text{for all }\mu\in Z.
    \end{aligned}\right.
\end{equation}
The weak forms of the Dirichlet--Neumann and Neumann--Neumann methods can be given in a similar form, but are not stated here since they are the same as their elliptical counterparts; see~\cite[Chapter 1.3]{quarteroni}.

\section{Time-dependent Steklov--Poincar\'e operators}\label{sec:SP}
The goal is now to reformulate the transmission problem and the space-time domain decomposition methods to problems and methods given on the interface $\Gamma\times\R$. In order to do so, we first introduce the time-dependent Steklov--Poincaré operators $S_i:Z\rightarrow Z^*$ and the functionals $\chi_i\in Z^*$ by
\begin{displaymath}
\langle S_i\eta, \mu\rangle=a_i(F_i\eta, F_i\mu)\quad\text{and}
\quad\langle\chi_i, \mu\rangle=(f_i, F_i\mu)_{L^2(\Omega_i\times\R)}-a_i(G_if_i, F_i\mu),
\end{displaymath}
respectively. We also write $S=S_1+S_2$ and $\chi=\chi_1+\chi_2$. We can then introduce the weak Steklov--Poincaré equation $S\eta=\chi$ in $Z^*$ or, equivalently,
\begin{equation}\label{eq:speq}
	   \sum_{i=1}^2 \langle S_i\eta, \mu\rangle=\sum_{i=1}^2 \langle\chi_i, \mu\rangle\quad \textrm{for all } \mu\in Z.
\end{equation}
\begin{remark}
The Steklov--Poincaré operators do not depend on the choice of extension $F_i$. For an arbitrary extension $R_i:Z\rightarrow W_i$ such that $T_iR_i\mu=\mu$ we have, by \cref{eq:Wi0}, that $F_i\mu-R_i\mu\in W_i^0$.  Combining this with \cref{eq:Fi} implies that
\begin{displaymath}
\langle S_i\eta, \mu\rangle=a_i(F_i\eta, F_i\mu)=a_i(F_i\eta, F_i\mu-R_i\mu)+a_i(F_i\eta, R_i\mu)=a_i(F_i\eta, R_i\mu).
\end{displaymath}
\end{remark}

The Steklov--Poincaré operators have the corresponding properties as the bilinear forms $a_i$ in \cref{lemma:ai}.
\begin{lemma}\label{lemma:sp}
Suppose that \cref{ass:domains,ass:f} hold. Then the time-dependent Steklov--Poincaré operator $S_i$ is bounded and satisfies the coercivity bound
\begin{equation}\label{eq:weakcoer}
\langle S_i\eta, \eta\rangle\geq c\|F_i\eta\|_{L^2(\R)\tensor H^1(\Omega_i)}^2 \quad \text{for all }\eta\in Z.
\end{equation}
Moreover, $S_i$ is coercive-equivalent in $Z$ with the isomorphism $\Hc_\Gamma^\varphi$ for $\varphi>0$ small enough. That is,
\begin{displaymath}
\langle S_i\eta, \Hc_\Gamma^\varphi\eta\rangle\geq c\|\eta\|_Z^2 \quad \text{for all }\eta\in Z.
\end{displaymath}
Analogous results hold for $S$.
\end{lemma}
\begin{proof}
Throughout the proof, let $\eta,\mu\in Z$ be arbitrary elements. The boundedness of $S_i$ is proved by \cref{lemma:ai,lemma:Fi}, since
\begin{displaymath}
|\langle S_i\eta, \mu\rangle| = |a_i(F_i\eta, F_i\mu)|\leq C\|F_i\eta\|_{W_i}\|F_i\mu\|_{W_i}\leq C\|\eta\|_Z\|\mu\|_Z.
\end{displaymath}
We also obtain the coercivity bound
\begin{displaymath}
\langle S_i\eta, \eta\rangle=a_i(F_i\eta, F_i\eta)\geq c\|F_i\eta\|_{L^2(\R)\tensor H^1(\Omega_i)}^2,
\end{displaymath}
using \cref{lemma:ai}. From \cref{eq:hilbertphi} we have
\begin{displaymath}
T_i(F_i\Hc_{\Gamma}^\varphi-\Hc_i^\varphi F_i)\eta=T_iF_i\Hc_{\Gamma}^\varphi\eta-T_i\Hc_i^\varphi F_i\eta
=\Hc_{\Gamma}^\varphi\eta-\Hc_\Gamma^\varphi T_iF_i\eta=0
\end{displaymath}
and therefore $(F_i\Hc_{\Gamma}^\varphi-\Hc_i^\varphi F_i)\eta\in W_i^0$ by \cref{eq:Wi0}. This together with \cref{eq:Fi} implies that
\begin{align*}
 a_i(F_i\eta, F_i\Hc_{\Gamma}^\varphi\eta) &=
                  a_i\bigl(F_i\eta, (F_i\Hc_{\Gamma}^\varphi-\Hc_i^\varphi F_i)\eta\bigr)+a_i(F_i\eta, \Hc_i^\varphi F_i\eta)\\
                  &=a_i(F_i\eta, \Hc_i^\varphi F_i\eta).
\end{align*}
Then \cref{lemma:tracei,lemma:ai} give that
\begin{displaymath}
\langle S_i\eta, \Hc_\Gamma^\varphi\eta\rangle=a_i(F_i\eta, \Hc_i^\varphi F_i\eta)\geq c\|F_i\eta\|_{W_i}^2\geq c\|T_iF_i\eta\|_Z^2=c\|\eta\|_Z^2.
\end{displaymath}
The bounds for $S$ follow by summing the bounds for $S_1$ and $S_2$. 
\qed
\end{proof}
From \cref{lemma:density} we have the Gelfand triple
\begin{displaymath}
Z\hookrightarrow L^2(\Gamma\times\R)\cong L^2(\Gamma\times\R)^*\hookrightarrow Z^*.
\end{displaymath}
We introduce the Riesz isomorphism $J: L^2(\Gamma\times\R)\rightarrow L^2(\Gamma\times\R)^*$ and note that the Gelfand triple implies the identity
\begin{equation}\label{eq:riesz}
\langle J\eta, \mu\rangle=(\eta, \mu)_{L^2(\Gamma\times\R)}\quad\text{ for all } \eta\in L^2(\Gamma\times\R),\,\mu\in Z.
\end{equation}
We trivially also have the bounds
\begin{equation}\label{eq:rieszbound}
\langle J\eta, \mu\rangle\leq\|\eta\|_{L^2(\Gamma\times\R)}\|\mu\|_{L^2(\Gamma\times\R)}\quad\text{ for all } \eta\in L^2(\Gamma\times\R),\,\mu\in Z
\end{equation}
and
\begin{equation}\label{eq:rieszcoer}
\langle J\eta, \eta\rangle\geq\|\eta\|_{L^2(\Gamma\times\R)}^2\quad\text{ for all } \eta\in Z.
\end{equation}
\begin{theorem}\label{thm:speq}
Suppose that \cref{ass:domains,ass:f} hold and let $s\geq 0$. Then the operators $sJ+S_i,sJ+S: Z\rightarrow Z^*$ are isomorphisms. In particular, there exists a unique solution to the weak Steklov--Poincaré equation \cref{eq:speq}.
\end{theorem}
\begin{proof}
For simplicity we assume that $s=0$. Then $S_i$ is bijective if and only if for any $z\in Z^*$ there exists a unique solution $\eta\in Z$ to
\begin{displaymath}
\langle S_i\eta, \mu\rangle=\langle z, \mu\rangle\quad \textrm{ for all } \mu\in Z.
\end{displaymath}
Since $\Hc_\Gamma^\varphi: Z\rightarrow Z$ is an isomorphism this is equivalent to
\begin{equation}\label{eq:Sibij}
\langle S_i\eta, \Hc_\Gamma^\varphi\mu\rangle=\langle z, \Hc_\Gamma^\varphi\mu\rangle\quad \textrm{ for all } \mu\in Z.
\end{equation}
Note that $\mu\mapsto \langle z, \Hc_\Gamma^\varphi\mu\rangle$ is also an element in $Z^*$. Moreover, by \cref{lemma:sp} the bilinear form $(\eta, \mu)\mapsto\langle S_i\eta, \Hc_\Gamma^\varphi\mu\rangle$ is bounded and coercive in $Z$. Therefore, by the Lax--Milgram lemma~\cite[Chapter 6, Theorem 6]{lax} there exists a unique solution $\eta$ to~\cref{eq:Sibij}. It follows from~\cref{lemma:sp} that $S_i$ and $S_i^{-1}$ are bounded and therefore $S_i$ is an isomorphism. Since we have \cref{eq:rieszbound,eq:rieszcoer} the proof can be generalized to any $s\geq 0$. The statement for $S$ is proven similarly to $S_i$.
\qed
\end{proof}

\begin{lemma}\label{lemma:tpspeq}
Suppose that \cref{ass:domains,ass:f} hold. The weak transmission problem and the weak Steklov--Poincaré equation are equivalent in the following way: If $(u_1, u_2)$ solves \cref{eq:weaktran} then $\eta=T_iu_i$ solves \cref{eq:speq}. Conversely, if $\eta$ solves \cref{eq:speq} then $(u_1, u_2)=(F_1\eta+G_1f_1, F_2\eta+G_2f_2)$ solves \cref{eq:weaktran}.
\end{lemma}
The proof of \cref{lemma:tpspeq} is immediate after writing out the definitions of the Steklov--Poincaré operators $S_i$, $i=1,2$. 

We will now introduce the iterations on the interface $\Gamma\times\R$ and show equivalence to the corresponding domain decomposition methods. The first one is the weak Peaceman--Rachford iteration, which is given by finding $(\eta^n_1,\eta^n_2)\in Z\times Z$, for $n=1,2,\ldots,$ such that
\begin{equation}\label{eq:pr}
\left\{\begin{aligned}
		 \langle(sJ+S_1)\eta_1^{n+1}-\chi_1, \mu\rangle&=\langle(sJ-S_2)\eta_2^n+\chi_2, \mu\rangle\quad&\text{for all }\mu\in Z,\\
                  \langle(sJ+S_2)\eta_2^{n+1}-\chi_2, \mu\rangle&=\langle(sJ-S_1)\eta_1^{n+1}+\chi_1, \mu\rangle\quad&\text{for all }\mu\in Z.
\end{aligned}\right.
\end{equation}
Here, $\eta_2^0\in Z$ is an initial guess. 
\begin{lemma}\label{lemma:rrpr}
Suppose that \cref{ass:domains,ass:f} hold. Then the weak forms of the Robin--Robin method and the Peaceman--Rachford iteration are equivalent in the following way: If $(u_1^n, u_2^n)_{n\geq 1}$ solves \cref{eq:robin} then $(\eta_1^n, \eta_2^n)_{n\geq 1}$, defined as $\eta_i^n=T_iu^n_i$, solves \cref{eq:pr} with $\eta_2^0=T_2u^0_2$. Conversely, if $(\eta_1^n, \eta_2^n)_{n\geq 1}$ solves \cref{eq:pr} then $(u_1^n, u_2^n)_{n\geq 1}$, defined as $u_i^n=F_i\eta_i^n+G_if_i$, solves \cref{eq:robin} with $u_2^0=F_2\eta_2^0+G_2f_2$.
\end{lemma}
\begin{proof}
Let $(u_1^n, u_2^n)_{n\geq 1}\subset W_1\times W_2$  be a weak Robin--Robin approximation solving~\cref{eq:robin} and define $\eta_i^n=T_iu^n_i\in Z$. This together with the first and third equation in~\cref{eq:robin} yields the identification $u_i^n=F_i\eta_i^n+G_if_i$.
Entering these observations into the second and fourth equation in~\cref{eq:robin} gives the equalities 
\begin{equation}\label{eq:inter}
   \begin{aligned}
            &s(\eta_1^{n+1}, \mu)_{L^2(\Gamma\times\R)}+a_1(F_1\eta_1^{n+1}+G_1f_1, F_1\mu)-( f_1, F_1\mu)_{L^2(\Omega_1\times\R)}\\
            &=s(\eta_2^n, \mu)_{L^2(\Gamma\times\R)}-a_2(F_2\eta_2^n+G_2f_2, F_2\mu)+( f_2, F_2\mu)_{L^2(\Omega_2\times\R)},\quad\text{and}\\[3pt]
            &s(\eta_2^{n+1}, \mu)_{L^2(\Gamma\times\R)}+a_2(F_2\eta_2^{n+1}+G_2f_2, F_2\mu)-( f_2, F_2\mu)_{L^2(\Omega_2\times\R)}\\
            &=s(\eta_1^{n+1}, \mu)_{L^2(\Gamma\times\R)}-a_1(F_1\eta_1^{n+1}+G_1f_1, F_1\mu)+( f_1, F_1\mu)_{L^2(\Omega_1\times\R)}
    \end{aligned}
\end{equation}
for all $\mu\in Z$. By the definitions of $S_i$ and $\chi_i$, the equalities~\cref{eq:inter} are equivalent to~\cref{eq:pr}, i.e., $(\eta_1^n, \eta_2^n)_{n\geq 1}$ solves~\cref{eq:pr} with $\eta_2^0=T_2u^0_2$. 

Conversely, let $(\eta_1^n, \eta_2^n)_{n\geq 1}\subset Z\times Z$ be a weak Peaceman--Rachford approximation solving~\cref{eq:pr} and define $u_i^{n}=F_i\eta_i^{n}+G_if_i$. This definition ensures that the first and third equation in~\cref{eq:robin} are fulfilled, as well as $T_iu_i^{n}=\eta_i^{n}$. As already observed, one Peaceman--Rachford iterate~\cref{eq:pr} is equivalent to~\cref{eq:inter}. Inserting the definition of $u_i^n$ into~\cref{eq:inter} then yields that the second and fourth equation in~\cref{eq:robin}, i.e.,  $(u_1^n, u_2^n)_{n\geq 1}$ solves \cref{eq:robin} with $u_2^0=F_2\eta_2^0+G_2f_2$.
\qed
\end{proof}
\begin{corollary}\label{cor:rrpr}
Suppose that \cref{ass:domains,ass:f} hold. Then the weak forms of the Robin--Robin method and the Peaceman--Rachford iteration are well defined, that is, \cref{eq:robin,eq:pr} have unique solutions at each iteration.
\end{corollary}
\begin{proof}
It follows from \cref{thm:speq} that \cref{eq:pr} has a unique solution at each iteration and, by \cref{lemma:rrpr}, so does \cref{eq:robin}. \qed
\end{proof}
The theory above applies to the weak forms of the Dirichlet--Neumann and Neumann--Neumann methods as well. To this end, we introduce two preconditioned fixed point iterations on the interface $\Gamma\times \R$. The first one is given by finding $\eta^n\in Z$, for $n=1,2,\ldots,$ such that
\begin{equation}\label{eq:pfp1}
 \langle S_2\eta^{n+1}, \mu\rangle=\langle S_2\eta^n, \mu\rangle+s_0\langle -S\eta^n+\chi, \mu\rangle \quad \text{ for all } \mu\in Z.
\end{equation}
Here, $\eta^0$ is an initial guess and $s_0\in (0, 1)$ is a method parameter. The second one is given by finding $(\lambda_1^n, \lambda_2^n, \eta^n)\in Z\times Z\times Z$, for $n=1,2,\ldots,$ such that
\begin{equation}\label{eq:pfp2}
\left\{\begin{aligned} 
		\langle S_i\lambda_i^{n+1}, \mu\rangle&=\langle -S\eta^n+\chi, \mu\rangle &&\text{ for all } \mu\in Z,\, i=1, 2,\\
		\eta^{n+1}&=\eta^n+s_1\lambda_1^{n+1}+s_2\lambda_2^{n+1}.&&
\end{aligned}\right.
\end{equation}
Again, $\eta^0$ is an initial guess and $s_1, s_2>0$ are method parameters. 
\begin{lemma}\label{lemma:dnnn}
Suppose that \cref{ass:domains,ass:f} hold. The weak Dirichlet--Neumann method is equivalent to the interface iteration~\cref{eq:pfp1} and the weak Neumann--Neumann method is equivalent to the interface iteration~\cref{eq:pfp2}.
\end{lemma}
Here, the definition of the methods being equivalent and the proofs thereof follow in the same fashion as for~\cref{lemma:rrpr}. By \cref{thm:speq,lemma:dnnn}, we now trivially have the following existence and uniqueness results.
\begin{corollary}
Suppose that \cref{ass:domains,ass:f} hold. The weak forms of the Dirichlet--Neumann and Neumann--Neumann methods, and their corresponding interface iterations \cref{eq:pfp1} and \cref{eq:pfp2}, respectively, are all well defined.
\end{corollary}
The standard convergence theories of the Dirichlet--Neumann and Neumann--Neumann methods~\cite[Chapter 4]{quarteroni} do not apply for the heat equation since the time-dependent Steklov--Poincaré operators are not symmetric. In the following sections we instead prove convergence results for Robin--Robin-type methods.
    
\section{Convergence of the Robin--Robin method}\label{sec:conv}
The variational framework appears to be too general for analyzing the convergence of the Peaceman--Rachford iteration. Therefore, we introduce the Steklov--Poincaré operators as affine unbounded operators on $L^2(\Gamma\times\R)$ before we prove that the iteration converges. To this end, let 
\begin{align*}
        D(\mathcal{S}_i)&=\{\eta\in Z: S_i\eta-\chi_i\in L^2(\Gamma\times\R)^*\},\\
        D(\mathcal{S})&=\{\eta\in Z: S\eta-\chi\in L^2(\Gamma\times\R)^*\}
\end{align*}
and define the unbounded affine operators
\begin{align*}
	\mathcal{S}_i&:D(\mathcal{S}_i)\subseteq L^2(\Gamma\times\R)\rightarrow L^2(\Gamma\times\R): \eta\mapsto J^{-1}(S_i\eta-\chi_i),\\
	\mathcal{S}&:D(\mathcal{S})\subseteq L^2(\Gamma\times\R)\rightarrow L^2(\Gamma\times\R): \eta\mapsto J^{-1}(S\eta-\chi).
\end{align*}
In $L^2(\Gamma\times\R)$ the Steklov--Poincaré equation is to find $\eta\in D(\mathcal{S})$ such that
\begin{equation}\label{eq:spl2}
	  \mathcal{S}\eta=0,
\end{equation}
and the Peaceman--Rachford iteration takes the following form: For each $n=1, 2,\dots$ find $(\eta_1^n, \eta_2^n)\in D(\mathcal{S}_1)\times D(\mathcal{S}_2)$ such that
\begin{equation}\label{eq:prl2}
  \left\{\begin{aligned} (sI+\mathcal{S}_1)\eta_1^{n+1}&=(sI-\mathcal{S}_2)\eta_2^n,\\
        (sI+\mathcal{S}_2)\eta_2^{n+1}&=(sI-\mathcal{S}_1)\eta_1^{n+1}.
\end{aligned}\right.
\end{equation}
Here, $\eta_2^0\in D(\mathcal{S}_2)$ is a given initial guess. We now verify that~\cref{eq:spl2} is indeed a restriction of the weak Steklov--Poincaré equation~\cref{eq:speq}.
\begin{lemma}\label{lemma:spvl2}
Suppose that \cref{ass:domains,ass:f} hold. If $\eta\in D(\mathcal{S})$ solves the $L^2$-Steklov--Poincaré equation~\cref{eq:spl2}, then $\eta$ also solves the weak Steklov--Poincaré equation~\cref{eq:speq}.
\end{lemma}
\begin{proof}
From~\cref{eq:spl2} we have
\begin{displaymath}
(J^{-1}(S\eta-\chi), \mu)_{L^2(\Gamma\times\R)}=0\quad\text{ for all }\mu\in L^2(\Gamma\times\R).
\end{displaymath}
Therefore, we get by~\cref{eq:riesz} that
\begin{displaymath}
\langle S\eta-\chi, \mu \rangle=(J^{-1}(S\eta-\chi), \mu)_{L^2(\Gamma\times\R)}=0
\end{displaymath}
for all $\mu\in Z$.
\qed
\end{proof}
A similar result holds for the Peaceman--Rachford iteration. The proof is left out since it is similar to the proof of~\cref{lemma:spvl2}. 
\begin{lemma}\label{lemma:prl2weak}
Suppose that \cref{ass:domains,ass:f} hold. If $(\eta_1^n, \eta_2^n)_{n\geq 1}$ solves the $L^2$-Peaceman--Rachford iteration~\cref{eq:prl2} with $\eta_2^0\in D(\mathcal{S}_2)$, then $(\eta_1^n, \eta_2^n)_{n\geq 1}$ also solves the weak Peaceman--Rachford iteration~\cref{eq:pr} with the same initial guess.
\end{lemma}
\begin{lemma}
Suppose that \cref{ass:domains,ass:f} hold. Then $\mathcal{S}_i,\, i=1, 2$, satisfy the monotonicity property
\begin{equation}\label{eq:coerl2} 
(\mathcal{S}_i\eta-\mathcal{S}_i\mu, \eta-\mu)_{L^2(\Gamma\times\R)}\geq 
	c\|F_i(\eta-\mu)\|^2_{L^2(\R)\tensor H^1(\Omega_i)}\quad \text{ for all }\eta, \mu\in D(\mathcal{S}_i).
\end{equation}
Moreover, for any $s\geq0$ the operators $sI+\mathcal{S}_i: D(\mathcal{S}_i)\rightarrow L^2(\Gamma\times\R),\, i=1, 2$, are bijective. Similar results hold for $\mathcal{S}$. In particular, there exists a unique solution to \cref{eq:spl2} and the iteration \cref{eq:prl2} is well defined.
\end{lemma}
\begin{proof}
The monotonicity follows from~\cref{eq:riesz,eq:weakcoer}, since
\begin{align*}
      (\mathcal{S}_i\eta-\mathcal{S}_i\mu, \eta-\mu)_{L^2(\Gamma\times\R)}&=\langle (S_i\eta-\chi_i)-(S_i\mu-\chi_i), \eta-\mu\rangle\\
             &=\langle S_i(\eta-\mu), \eta-\mu\rangle\geq c\|F_i(\eta-\mu)\|^2_{L^2(\R)\tensor H^1(\Omega_i)}
\end{align*}
for all $\eta, \mu\in D(\mathcal{S}_i)$. Let $\mu\in L^2(\Gamma\times\R)$ be arbitrary. Then $\chi_i+J\mu\in Z^*$ and, by~\cref{thm:speq}, there exists a unique $\eta\in Z$ such that $(sJ+S_i)\eta=\chi_i+J\mu$ in $Z^*$. Rearranging yields that $S_i\eta-\chi_i=J(\mu-s\eta)\in L^2(\Gamma\times\R)^*$, i.e.,  $\eta\in D(\mathcal{S}_i)$ with
\begin{displaymath}
            (sI+\mathcal{S}_i)\eta=s\eta+J^{-1}(S_i\eta-\chi_i)=\mu.
\end{displaymath}
Thus, we have shown that $(sI+\mathcal{S}_i)$ is bijective. The proof for $\mathcal{S}$ is similar and is therefore left out.
\qed
\end{proof}

In order to proceed with the convergence analysis, we require the following regularity of the solution to the weak heat equation~\cref{eq:weak}. 
\begin{assumption}\label{ass:regularity}
The functionals  
\begin{displaymath}
\mu\mapsto a_i(\restr{u}{\Omega_i\times\R}, F_i\mu)-(f_i, F_i\mu)_{L^2(\Omega_i\times\R)},\quad{i=1,2},
\end{displaymath}
are elements in $L^2(\Gamma\times\R)^*$, where $u\in W$ is the solution to~\cref{eq:weak}.
\end{assumption}
\begin{remark}
The assumption is somewhat implicit, but can be interpreted as the solution having a normal derivative $\nabla u\cdot \nu_i$ on the space-time interface belonging to $L^2(\Gamma\times\R)$. This holds, for instance, if the solution satisfies the additional regularity 
\begin{displaymath}
u\in W\cap L^2\bigl(\R, H^2(\Omega)\bigr)\quad\text{or}\quad u\in W\cap L^2\bigl(\R,  C^1(\bar{\Omega})\bigr). 
\end{displaymath}
To see this, first observe that a Lipschitz manifold $\partial\Omega_i$ has a normal vector $\nu_i$ in $L^\infty(\partial\Omega_i)^d$. The additional regularity of~$u$ then yields that each term $\partial_j u\,(\nu_i)_j$ of the normal derivative becomes an element in $L^2(\Gamma\times\R)$. 
\end{remark}
\begin{lemma}\label{lemma:friedrich}
Suppose that \cref{ass:domains,ass:f,ass:regularity} hold. If $\eta$ is the solution to the $L^2$-Steklov--Poincaré equation~\cref{eq:spl2} then $\eta\in D(\mathcal{S}_1)\cap D(\mathcal{S}_2)$.
\end{lemma}
\begin{proof}
Let $\eta\in D(\mathcal{S})$ be the solution to~\cref{eq:spl2}. By~\cref{lemma:tpspeq,lemma:tranequiv}, we have the identification $\restr{u}{\Omega_i\times\R}=F_i\eta+G_if_i$, where $u\in W$ is the solution to the weak heat equation~\cref{eq:weak}. \cref{ass:regularity} then yields that 
 \begin{align*}
            \mu\mapsto\langle S_i\eta-\chi_i, \mu\rangle&=a_i(F_i\eta, F_i\mu)+a_i(G_if_i, F_i\mu)-(f_i, F_i\mu)_{L^2(\Omega_i\times\R)}\\
            &=a_i(\restr{u}{\Omega_i\times\R}, F_i\mu)-(f_i, F_i\mu)_{L^2(\Omega_i\times\R)}
\end{align*}
is an element in $L^2(\Gamma\times\R)^*$, i.e., $\eta\in D(\mathcal{S}_i)$ for $i=1,2$.
\qed
\end{proof}
 
\begin{lemma}\label{lemma:prlimit}
Suppose that \cref{ass:domains,ass:f,ass:regularity} hold. Let $\eta$ be the solution to the $L^2$-Steklov--Poincaré equation~\cref{eq:spl2} and $(\eta_1^n, \eta_2^n)_{n\geq 1}$ be the $L^2$-Peaceman--Rachford iterates~\cref{eq:prl2} with $\eta_2^0\in D(\mathcal{S}_2)$. Then we have the limit
\begin{equation}\label{eq:prlimit}
 (\mathcal{S}_i\eta^{n+1}_i-\mathcal{S}_i \eta, \eta^{n+1}_i -\eta)_{L^2(\Gamma\times\R)}\rightarrow 0,
\end{equation}
 as $n$ tends to infinity.
\end{lemma}
\cref{lemma:prlimit} is a consequence of the abstract result~\cite[Proposition 1]{lionsmercier}. The latter requires the monotonicity~\cref{eq:coerl2} and the fact that the solution to~\cref{eq:spl2} satisfies $\eta\in D(\mathcal{S}_1)\cap D(\mathcal{S}_2)$, which follows from~\cref{lemma:friedrich}. A simpler proof of~\cref{lemma:prlimit} can be found in~\cite[Lemma~8.8]{EHEE22}.

We are now in a position to prove that the Robin--Robin method converges.
\begin{theorem}\label{thm:rrconv}
Suppose that \cref{ass:domains,ass:f,ass:regularity} hold. Let $u$ be the solution to the weak heat equation~\cref{eq:weak} and $\eta$ be the solution to the $L^2$-Steklov--Poincaré equation~\cref{eq:spl2}. The iterates $(\eta_1^n, \eta_2^n)_{n\geq 1}$ of the $L^2$-Peaceman--Rachford iteration~\cref{eq:prl2} converges to $\eta$, i.e.,
\begin{equation*}\label{eq:prconv}
\|\eta^n_1-\eta\|_{L^2(\R)\tensor \Lambda}+\|\eta^n_2-\eta\|_{L^2(\R)\tensor \Lambda}\rightarrow 0,
\end{equation*}
as $n$ tends to infinity. Moreover the weak Robin--Robin approximation $(u_1^n, u_2^n)_{n\geq 1}$ converges to $(u_1, u_2)=(\restr{u}{\Omega_1\times\R}, \restr{u}{\Omega_2\times\R})$, i.e.,
\begin{equation*}\label{eq:rrconv}
\|u^n_1- u_1\|_{L^2(\R)\tensor H^1(\Omega_1)}+\|u^n_2-u_2\|_{L^2(\R)\tensor H^1(\Omega_2)}\rightarrow 0,
\end{equation*}
as $n$ tends to infinity.
\end{theorem}
\begin{proof}
From \cref{eq:prlimit,eq:coerl2} we obtain
\begin{align*}
            &\|\eta^n_1-\eta\|^2_{L^2(\R)\tensor \Lambda}+\|\eta^n_2-\eta\|^2_{L^2(\R)\tensor \Lambda}\\
            &\quad=\|T_1F_1(\eta^n_1-\eta)\|^2_{L^2(\R)\tensor \Lambda}+\|T_2F_2(\eta^n_2-\eta)\|^2_{L^2(\R)\tensor \Lambda}\\
            &\quad\leq C\bigl(\|F_1(\eta^n_1-\eta)\|^2_{L^2(\R)\tensor H^1(\Omega_1)}+\|F_2(\eta^n_2-\eta)\|^2_{L^2(\R)\tensor H^1(\Omega_2)}\bigr)\\
            &\quad\leq C\bigl((\mathcal{S}_1\eta^{n+1}_1-\mathcal{S}_1 \eta, \eta^{n+1}_1 -\eta)_{L^2(\Gamma\times\R)}
                +(\mathcal{S}_2\eta^n_2- \mathcal{S}_2 \eta, \eta^n_2 -\eta)_{L^2(\Gamma\times\R)}\bigr)\rightarrow 0,
\end{align*}
as $n$ tends to infinity. By~\cref{lemma:tpspeq,lemma:tranequiv,lemma:rrpr}, one has the identities
\begin{displaymath}
 (u_1, u_2)=(F_1\eta+G_1f_1, F_2\eta+G_2f_2)\quad\text{and}\quad (u_1^n, u_2^n)=(F_1\eta_1^n+G_1f_1, F_2\eta_2^n+G_2f_2).
\end{displaymath}
This together with the last bound in the limit above yields that 
\begin{align*}
            &\|u^n_1-u_1\|^2_{L^2(\R)\tensor H^1(\Omega_1)}+\|u^n_2-u_2\|^2_{L^2(\R)\tensor H^1(\Omega_2)}\\
            &\quad=\|F_1(\eta^n_1-\eta)\|^2_{L^2(\R)\tensor H^1(\Omega_1)}+\|F_2(\eta^n_2-\eta)\|^2_{L^2(\R)\tensor H^1(\Omega_2)}\rightarrow 0,
\end{align*}
as $n$ tends to infinity.
\qed
\end{proof}
\begin{remark}
    Note that the Robin--Robin method is unlikely to have a geometric convergence rate in the present continuous framework. This is indicated already in the elliptic case, by observing that the spatially discrete method has a convergence rate of the form $1-\mathcal{O}(\sqrt{h})$; see~\cite{gander06}. That is, a rate deteriorating to one as the spatial discretization parameter $h$ tends to zero.
\end{remark}

\section{Modified Robin--Robin method}\label{sec:convmod}
One can achieve a slightly stronger convergence result if one allows the evaluation of the (non-local) Hilbert transform in the method. Instead of the standard Robin--Robin method, we therefore also consider the method given by finding, for each $n=1,2,\dots$, functions $(u_1^n, u_2^n)\in W_1\times W_2$ such that
	\begin{equation}\label{eq:hrobin}
    \left\{\begin{aligned}
    	&a_1(u_1^{n+1}, v_1)=(f_1, v_1)_{L^2(\Omega_1\times\R)} & &\text{for all } v_1\in W_1^0,\\
            &a_1(u_1^{n+1}, F_1\Hc^\varphi_\Gamma\mu)-(f_1, F_1\Hc^\varphi_\Gamma\mu)_{L^2(\Omega_1\times\R)} & &\\
            &\quad+a_2(u_2^n, F_2\Hc^\varphi_\Gamma\mu)-(f_2, F_2\Hc^\varphi_\Gamma\mu)_{L^2(\Omega_2\times\R)} & &\\
            &\quad=s(T_{2}u_2^n-T_{1}u_1^{n+1}, \mu)_{L^2(\Gamma\times\R)} & &\text{for all }\mu\in Z,\\[5pt]
    	&a_2(u_2^{n+1}, v_2)=(f_2, v_2)_{L^2(\Omega_2\times\R)} & &\text{for all } v_2\in W_2^0,\\
    	&a_2(u_2^{n+1}, F_2\Hc^\varphi_\Gamma\mu)-(f_2, F_2\Hc^\varphi_\Gamma\mu)_{L^2(\Omega_2\times\R)} & &\\
    	&\quad+a_1(u_1^{n+1}, F_1\Hc^\varphi_\Gamma\mu)-(f_1, F_1\Hc^\varphi_\Gamma\mu)_{L^2(\Omega_1\times\R)} & & \\
    	&\quad=s(T_{1}u_1^{n+1}-T_{2}u_2^{n+1}, \mu)_{L^2(\Gamma\times\R)} & &\text{for all }\mu\in Z.
    \end{aligned}\right.
    \end{equation}
	Here, $u_2^0 \in W_2$ is an initial guess and $s, \varphi>0$ are suitably chosen method parameters. We refer to~\cref{eq:hrobin} as the weak formulation of the modified Robin--Robin method. Moreover, we consider the weak modified Peaceman--Rachford iteration, as finding functions $(\eta_1^n, \eta_2^n)\in Z\times Z$, for each $n=1,2,\dots$, such that
	\begin{equation}\label{eq:hpr}
	    \left\{\begin{aligned}
        &\langle(sJ+\Hc_\Gamma^{\varphi*} S_1)\eta_1^{n+1}-\Hc_\Gamma^{\varphi*}\chi_1, \mu\rangle & &\\
        &\quad=\langle(sJ-\Hc_\Gamma^{\varphi*} S_2)\eta_2^n+\Hc_\Gamma^{\varphi*}\chi_2, \mu\rangle& &\text{for all } \mu\in Z,\\
        &\langle(sJ+\Hc_\Gamma^{\varphi*} S_2)\eta_2^{n+1}-\Hc_\Gamma^{\varphi*}\chi_2, \mu\rangle & &\\
        &\quad=\langle(sJ-\Hc_\Gamma^{\varphi*} S_1)\eta_1^{n+1}+\Hc_\Gamma^{\varphi*}\chi_1, \mu\rangle & &\text{for all } \mu\in Z.
        \end{aligned}\right.
	\end{equation}
	Here, $\eta_2^0\in Z$ is an initial guess and $\Hc_\Gamma^{\varphi*}$ denotes the adjoint of $\Hc_\Gamma^\varphi$ in $Z$, i.e.
	\begin{displaymath}
	    \langle \Hc_\Gamma^{\varphi*}z, \mu\rangle=\langle z, \Hc_\Gamma^\varphi\mu\rangle.
	\end{displaymath}
	for all $z\in Z^*$ and $\mu\in Z$.
	\begin{lemma}\label{lemma:rrhpr}
        Suppose that \cref{ass:domains,ass:f} hold. Then the weak forms of the modified Robin--Robin method~\cref{eq:hrobin} and the modified Peaceman--Rachford iteration~\cref{eq:hpr} are equivalent.
    \end{lemma}
	The proof of~\cref{lemma:rrhpr} is similar to the proof of~\cref{lemma:rrpr} and is therefore left out. As in the case of the Robin--Robin method we restrict the Steklov--Poincaré operators to $L^2(\Gamma\times\R)$. Let
	\begin{align*}
        D(\mathcal{S}_i^\varphi)&=\{\eta\in Z: \Hc_\Gamma^{\varphi*}(S_i\eta-\chi_i)\in L^2(\Gamma\times\R)^*\},\\
        D(\mathcal{S}^\varphi)&=\{\eta\in Z: \Hc_\Gamma^{\varphi*}(S\eta-\chi)\in L^2(\Gamma\times\R)^*\}
    \end{align*}
    and define the unbounded affine operators
	\begin{align*}
	    \mathcal{S}_i^\varphi&:D(\mathcal{S}_i^\varphi)\subseteq L^2(\Gamma\times\R)\rightarrow L^2(\Gamma\times\R): \eta\mapsto J^{-1}\Hc_\Gamma^{\varphi*}(S_i\eta-\chi_i),\\
	    \mathcal{S}^\varphi&:D(\mathcal{S}^\varphi)\subseteq L^2(\Gamma\times\R)\rightarrow L^2(\Gamma\times\R): \eta\mapsto J^{-1}\Hc_\Gamma^{\varphi*}(S\eta-\chi).
	\end{align*}
	The $L^2$-formulation of the modified Peaceman--Rachford iteration is then the following: For each $n=1,2,\dots$ find $(\eta_1^{n+1}, \eta_2^{n+1})\in D(\mathcal{S}_1^\varphi)\times D(\mathcal{S}_2^\varphi)$ such that
	\begin{equation}\label{eq:hprl2}
	    \left\{\begin{aligned} (sI+\mathcal{S}_1^\varphi)\eta_1^{n+1}&=(sI-\mathcal{S}_2^\varphi)\eta_2^n,\\
        (sI+\mathcal{S}_2^\varphi)\eta_2^{n+1}&=(sI-\mathcal{S}_1^\varphi)\eta_1^{n+1}.
        \end{aligned}
        \right.
	\end{equation}
	Here $\eta_2^0 \in D(\mathcal{S}_2^\varphi)$ is an initial guess. We have an analogous result to~\cref{lemma:prl2weak} so that \cref{eq:hprl2} and \cref{eq:hpr} are equivalent for $\eta_0^n\in D(\mathcal{S}_2^\varphi)$. Moreover, we have the following result on the domain $D(\mathcal{S}_i^\varphi)$:
	\begin{lemma}
	    Suppose that \cref{ass:domains,ass:f} hold. Then
	    \begin{displaymath}
	        D(\mathcal{S}_i^\varphi)=D(\mathcal{S}_i),\quad i=1,2.
	    \end{displaymath}
	    If \cref{ass:regularity} also holds and $\eta$ is the solution to \cref{eq:spl2} then $\eta\in D(\mathcal{S}_1^\varphi)\cap D(\mathcal{S}_2^\varphi)$.
	\end{lemma}
	\begin{proof}
	    Since $\Hc_\Gamma^\varphi: L^2(\Gamma\times\R)\rightarrow L^2(\Gamma\times\R)$ is an isomorphism we have that $\Hc_\Gamma^{\varphi*}:L^2(\Gamma\times\R)^*\rightarrow L^2(\Gamma\times\R)^*$ also is an isomorphism. Therefore $D(\mathcal{S}_i^\varphi)=D(\mathcal{S}_i)$ and the remaining statement follows from \cref{lemma:friedrich}.
	    \qed
	\end{proof}
	The theorem below follows by the same argument as \cref{thm:rrconv}, the difference being that the operator $\Hc_\Gamma^{\varphi*}S_i$ is coercive in the stronger norm of $Z$; see~\cref{lemma:sp}. This in turn gives the stronger monotonicity property
    \begin{equation} 
        (\mathcal{S}_i^\varphi\eta-\mathcal{S}_i^\varphi\mu, \eta-\mu)_{L^2(\Gamma\times\R)}\geq 
	    c\|F_i(\eta-\mu)\|^2_Z\quad \text{ for all }\eta, \mu\in D(\mathcal{S}_i^\varphi).
    \end{equation}
	\begin{theorem}\label{thm:trrconv}
        Suppose that \cref{ass:domains,ass:f,ass:regularity} hold. Let $u$ be the solution to the weak heat equation~\cref{eq:weak} and $\eta$ be the solution to the $L^2$-Steklov--Poincaré equation~\cref{eq:spl2}. The iterates $(\eta_1^n, \eta_2^n)_{n\geq 1}$ of the $L^2$-modified Peaceman--Rachford iteration~\cref{eq:hprl2} converges to $\eta$, i.e.,
        \begin{equation*}\label{eq:tprconv}
            \|\eta^n_1-\eta\|_Z+\|\eta^n_2-\eta\|_Z\rightarrow 0,
        \end{equation*}
        as $n$ tends to infinity. Moreover, the weak modified Robin--Robin approximation\linebreak $(u_1^n, u_2^n)_{n\geq 1}$ converges to $(u_1, u_2)=(\restr{u}{\Omega_1\times\R}, \restr{u}{\Omega_2\times\R})$, i.e.,
        \begin{equation*}\label{eq:trrconv}
            \|u^n_1-u_1\|_{W_1}+\|u^n_2-u_2\|_{W_2}\rightarrow 0,
        \end{equation*}
        as $n$ tends to infinity.
    \end{theorem}

\bibliographystyle{plain}
\bibliography{references}

\end{document}